\documentclass[12pt]{amsart}
\usepackage{amsmath, amsthm, amssymb, geometry, url, mathtools, mathrsfs}
\usepackage{arydshln}
\usepackage{mathtools}

\usepackage[normalem]{ulem}

\usepackage{etaremune}
\usepackage{enumitem}

\usepackage{xcolor}
\usepackage[colorlinks = true,
            linkcolor = blue,
            urlcolor  = blue,
            citecolor = blue,
            anchorcolor = blue]{hyperref}

\evensidemargin \oddsidemargin
\allowdisplaybreaks

\author{Sara Lapan}
\address{Department of Mathematics\\ University of California\\ Riverside, CA 92521}
\email{sara.lapan@ucr.edu}

\author{Benjamin Linowitz}
\address{Department of Mathematics\\ Oberlin College\\ Oberlin, OH 44074}
\email{benjamin.linowitz@oberlin.edu}

\author{Jeffrey S. Meyer}
\address{Department of Mathematics\\ 
California State University\\ 
San Bernardino, CA 92407}
\email{jeffrey.meyer@csusb.edu}

\title{Universal systole bounds\\ for arithmetic locally symmetric spaces}

\usepackage{fancyhdr}

\fancypagestyle{plain}

\DeclareMathAlphabet{\curly}{U}{rsfs}{m}{n}

\DeclareMathOperator{\Ad}{Ad}

\DeclareMathOperator{\disc}{disc}

\DeclareMathOperator{\GL}{GL}

\DeclareMathOperator{\isom}{Isom}

\DeclareMathOperator{\Mat}{M}

\DeclareMathOperator{\N}{N}

\DeclareMathOperator{\Res}{Res}
\DeclareMathOperator{\Reg}{reg}

\DeclareMathOperator{\SL}{SL}
\DeclareMathOperator{\stab}{Stab}
\DeclareMathOperator{\SO}{SO}

\DeclareMathOperator{\sys}{sys}
\DeclareMathOperator{\MM}{M}

\DeclareMathOperator{\tr}{tr}

\DeclareMathOperator{\Vol}{Vol}

\DeclareMathOperator{\arccosh}{arccosh}

\makeatletter
\renewcommand*\env@matrix[1][*\c@MaxMatrixCols c]{%
  \hskip -\arraycolsep
  \let\@ifnextchar\new@ifnextchar
  \array{#1}}
\makeatother

\newtheorem{thm}{Theorem}[section]
\newtheorem{cor}[thm]{Corollary}

\newtheorem{prop}[thm]{Proposition}
\newtheorem{lem}[thm]{Lemma}
\theoremstyle{definition}

\newtheorem{rem}[thm]{Remark}

\newtheorem*{ack}{Acknowledgements}
\theoremstyle{remark}

\def\1{\mathbf{1}}
\def\disc{\mathrm{disc}}

\theoremstyle{plain}
\newtheorem{mainthm}{Theorem}

\theoremstyle{remark}

\theoremstyle{plain}

\def\C{\mathbf{C}}

\def\Q{\mathbf{Q}}

\def\R{\mathbf{R}}

\def\N{\mathbf{N}}
\def\Q{\mathbf{Q}}
\def\Z{\mathbf{Z}}

\def\1{\mathbf{1}}
\def\disc{\mathrm{disc}}

\newcommand{\abs}[1]{\left\vert#1\right\vert}

\makeatletter
\def\moverlay{\mathpalette\mov@rlay}
\def\mov@rlay#1#2{\leavevmode\vtop{%
   \baselineskip\z@skip \lineskiplimit-\maxdimen
   \ialign{\hfil$\m@th#1##$\hfil\cr#2\crcr}}}
\newcommand{\charfusion}[3][\mathord]{
    #1{\ifx#1\mathop\vphantom{#2}\fi
        \mathpalette\mov@rlay{#2\cr#3}
      }
    \ifx#1\mathop\expandafter\displaylimits\fi}
\makeatother

\makeatletter
\let\@@pmod\pmod
\DeclareRobustCommand{\pmod}{\@ifstar\@pmods\@@pmod}
\def\@pmods#1{\mkern4mu({\operator@font mod}\mkern 6mu#1)}
\makeatother

\numberwithin{equation}{section}
\begin{document}

\begin{abstract}
The systole of a closed Riemannian manifold is the minimal length of a non-contractible closed loop.
We give a uniform lower bound for the systole for large classes of  simple arithmetic locally symmetric orbifolds.
We establish new bounds for the translation length of semisimple $x\in\SL_n(\R)$ in terms of its associated Mahler measure.
We use these geometric methods to prove the existence of extensions of number fields in which fixed sets of primes have certain prescribed splitting behavior.
\end{abstract}

\maketitle 


\vspace{-2pc}
\section{Introduction}

The \textit{systole} of a closed Riemannian manifold is the minimal length of a non-contractible closed loop. 
In this paper, we investigate lower bounds for the systoles of simple arithmetic locally symmetric orbifolds.

To give some history and motivation, we begin by describing the situation in dimension 2.  
%
%
%
If $\Gamma \subset \SL_2(\R)$ is a lattice and $\mathbb{H}^2$ is the (real) hyperbolic plane, then $S:=\mathbb{H}^2/\Gamma$ is a hyperbolic surface (possibly with cusps or cone points).
An element $x\in\SL_2(\R)$ is \textit{hyperbolic} if it is diagonalizable with positive real eigenvalues. Let $\{a,a^{-1}\}$ (where $a>1$) denote the eigenvalues of $x$.
If $x\in \Gamma$, then it corresponds to a closed geodesic in $S$ with length 
\begin{equation}\label{eq:length_log}
\ell(x):=2\log|a|.
\end{equation}
It is well known (\cite[\S12.3]{MaR} or \cite[\S4]{NR}) that if $\Gamma$ is an arithmetic lattice derived from a quaternion algebra, 
then $a$ is an algebraic integer greater than 1 which is Galois conjugate to its inverse and for which all other Galois conjugates have norm 1, or in other words, $a$ is a \textit{Salem number}.
Conversely, Chinburg, Neumann, and Reid showed that every Salem number appears this way for some arithmetic hyperbolic surface.

The \textit{Mahler measure} of a monic polynomial $p(X)=\prod_{i=1}^n(X-a_i)\in \C[X]$ is the product $M(p):=\prod_{i=1}^n \max\{1,|a_i|\}$.
The Mahler measure is a fascinating function on polynomials that appears throughout algebra, geometry, and dynamics 
(see for example the surveys \cite{GH} and \cite{smyth}).
If $p_x(X)$ is the characteristic polynomial of $x$, then it follows that \eqref{eq:length_log} can be rewritten in terms of the Mahler measure of $p_x$,
\begin{equation}\label{eq:original_mahler_length}
\ell(x) = 2\log(M(p_x)).
\end{equation}

Equation \eqref{eq:original_mahler_length} provides a deep connection between length and the Mahler measure.
If $p$ is a cyclotomic polynomial, then $M(p)=1$.  
\textit{Lehmer's Problem} asks: is the Mahler measure of irreducible non-cyclotomic monic polynomials with integer coefficients bounded away from 1?  
The polynomial 
\begin{equation}\label{eq:lehmers_poly}
L(X) = X^{10}+X^9-X^7-X^6-X^5-X^4-X^3+X+1
\end{equation}
found by Lehmer (and now called Lehmer's polynomial) has Mahler measure  $M(L) \approx 1.1762808$. 
Currently $M(L)$ is the smallest known nontrivial Mahler measure and it is still an open question if in fact this is the smallest over all monic polynomials in $\Z[X]$.
It follows that Lehmer's problem implies the open \textit{Short Geodesic Problem}, which asks: is there a universal lower bound for the systoles of arithmetic hyperbolic surfaces?
%

In \cite{ERT}, Emery, Ratcliffe, and Tschantz extended the
identification \eqref{eq:original_mahler_length}  
to $n$-dimensional arithmetic hyperbolic orbifolds of simplest type.  
In particular, they were able to get systole bounds in terms of Salem numbers of bounded degree. 

%
%
%
%
%
%

Our goal in this paper is to generalize these results and establish lower bounds for broad classes of arithmetic locally symmetric spaces.
In order to achieve this goal we need a generalization of \eqref{eq:original_mahler_length} to semisimple $x\in\SL_n(\R)$.
In Section \ref{section:mahler}, we establish such a generalization with our Mahler-Length Bounds Theorem \ref{thm:sarasbounds} below.   


\begin{mainthm}[Mahler-Length Bounds]\label{thm:sarasbounds}
If $n\ge2$, $x\in \SL_n(\R)$ is semisimple, and $p_x$ denotes the characteristic polynomial of $x$, then
\begin{equation}\label{eq:sarasbounds}
2\sqrt{\frac{2}{n}}\log(M(p_x))
\leq\ell(x)\leq
 2 \log(M(p_x)).  
\end{equation}
Let $\{a_1,\ldots, a_n\}$ denote the eigenvalues of $x$. The lower bound is an equality exactly when there exists an $r\in\R$ such that for all $i$, either $|a_i|=r$ or $|a_i|^{-1}=r$,
and the upper bound is an equality when $\abs{a_i}=1$ for all but at most two of the $i$.
\end{mainthm}

\begin{rem}
Let $x$ be such that in  \eqref{eq:sarasbounds} the upper bound is an equality and at least one of its eigenvalues has modulus not equal to 1.  
Since $p_x$ has real coefficients, if $a_i$ is complex, its complex conjugate $\overline{a_i}$ is also an eigenvalue.
By our assumption on $x$, it has precisely two eigenvalues of modulus not equal to 1, and as the eigenvalues must multiply to 1, it follows that $x$ or $-x$ has a unique real eigenvalue greater than 1, say $\alpha$. 
When $p_x$ has integer coefficients, then $\alpha$ is a Salem number. 
\end{rem}

\begin{rem}
Our Mahler-Length Bounds Theorem \ref{thm:sarasbounds} is a generalization of \eqref{eq:original_mahler_length} to degree $n\ge2$ just as our Trace-Length Bounds Theorem in \cite{LLM} was a generalization of the relationship $\ell(x)=2\arccosh(\frac{\tr{x}}{2})$ to degree $n\ge2$.
\end{rem}

In Section \ref{section:universal_bounds} we use our Mahler-Length Bounds Theorem \ref{thm:sarasbounds} to establish lower bounds for the systoles of simple arithmetic orbifolds.

\begin{mainthm}[Universal Systole Bounds]\label{thm:universalbound}
If $N$ is a simple arithmetic orbifold with field of definition $k$ and universal cover $\mathfrak{X}$, then 
\begin{equation}\label{eq:gen_systole_bound}
\sys(N,g)\ge \frac{2\sqrt{2}}{\sqrt{n}}f(n),
\end{equation}
where $\sys(N,g)$ denotes the systole of $N$ (as measured relative to the subspace metric \eqref{metric_eq}), $n=|k:\Q| \dim(\isom(\mathfrak{X}))$, and 
\begin{align}
f(n) = 
\begin{cases} 
\log(M(L))> 0.16235& \mbox{ if } n\le 20,\\
\frac{1}{4}\left(\frac{\log\log n}{\log n}\right)^3& \mbox{ if } n>20,
\end{cases}
\end{align}
where $L(X) = X^{10}+X^9-X^7-X^6-X^5-X^4-X^3+X+1$ is Lehmer's polynomial.
\end{mainthm}

To prove this result, we show in Proposition \ref{prop:rep_char_poly} that the special linear group is ``universal'' in the sense that if $N\cong \mathfrak{X}/\Gamma$ where $\Gamma \subset \isom(\mathfrak{X})$ is arithmetic, then $\Gamma$ injects into a special linear group of explicitly bounded degree and each matrix in the image of this injection has a characteristic polynomial with integer coefficients. 

\begin{rem}
Fix a globally symmetric space $\mathfrak{X}$ of noncompact type.
When $N\cong \mathfrak{X}/\Gamma$ where $\Gamma$ is a \textit{principal arithmetic lattice} in the sense of \cite[3.4]{P}, then as the volume of $N$ increases, Prasad's volume formula \cite[3.11]{P} implies the degree $|k:\Q|$ increases, and hence $n$ in \eqref{eq:gen_systole_bound} increases, 
thereby decreasing the systole bound.  
Hence Theorem \ref{thm:universalbound} implies there is a fixed universal lower systole bound for orbifolds of bounded volume, but as volume is allowed to increase, this bound goes to 0.  
\end{rem}


%

In \cite[Theorem B]{LLM} we showed that every simple arithmetic $N$ has a finite degree cover $N'$ such that $\sys(N')\ge c_1\log(\Vol(N'))+c_2$ for certain constants $c_1$ and $c_2$.
In Section \ref{section:large_systole} we prove that certain covers will always attain systoles with order of magnitude log of volume.

\begin{mainthm}[Simple Arithmetic Manifolds with Large Systoles]\label{thm:magnitude}
If $N$ is a compact simple arithmetic manifold, then $N$ has a finite sheeted cover $N'$ whose systole has order of magnitude $\log(\Vol(N'))$.
\end{mainthm}


%
%


A particularly interesting specialization of Theorem \ref{thm:universalbound} is to the case of special linear orbifolds derived from central simple algebras.
In Section \ref{section:speciallinearbackground} we give the construction of special linear orbifolds derived from central simple algebras and give an alternate proof of Theorem \ref{thm:systolebounds} below.

\begin{mainthm}[Standard Special Linear Short Geodesic Bounds]\label{thm:systolebounds}
If $M$ is a standard special linear orbifold of degree $n$, $n\ge 2$, which is derived from a central simple algebra, then
\begin{align}
\sys(M)\ge \frac{2\sqrt{2}}{\sqrt{n}}f(n),
\end{align}
where $f(n)$ is as in Theorem \ref{thm:universalbound}.
\end{mainthm}

%
%

Interestingly, the results of this paper can be used to prove a result from algebraic number theory: for any fixed totally real number field $k$ and finite set $S$ of primes of $k$, there exists a prime degree field extension of $k$ in which no prime of $S$ splits completely and whose norm of relative discriminant may be bounded from above. Although such a result follows from quantitative refinements of the (deep) Grunwald-Wang theorem (see \cite{effectiveGW} and \cite{Wang2}), we are able to give a short, geometrically-inspired proof.
%


\begin{mainthm}[Existence of Number Fields with Prescribed Arithmetic Properties]\label{thm:effectivegw}
Let $p>2$ be a prime number, $k$ be a totally real number field and $\mathfrak{p}_1,\dots, \mathfrak p_s$ be primes of $k$. There exists a field extension $L/k$ of degree $p$ in which none of the primes $\mathfrak p_i$ split completely and which has norm of relative discriminant bounded above by $c d_k^{c'} \left(\prod_{i=1}^s N_{k/\Q}(\mathfrak p_i)\right)^{c''}$ where $c,c',c''$ are positive constants depending only on $p$ and the degree over $\Q$ of $k$.
\end{mainthm}

\begin{ack}
The authors would like to thank Ralf Spatzier for his useful comments regarding the proof of Proposition \ref{systoleupperbound}. 
The work of the second author is partially supported by NSF Grant Number DMS-1905437. 
The third author acknowledges support from U.S. National Science Foundation grants DMS 1107452, 1107263, 1107367 ``RNMS: Geometric Structures and Representation Varieties'' (the GEAR Network).
\end{ack}

\section{Mahler Measure and Translation Length Inequalities - Theorem \ref{thm:sarasbounds}}\label{section:mahler}

Throughout this section, we will let $p(X)=\prod_{i=1}^n(X-a_i)\in \R[X]$ denote a monic degree $n$ polynomial.
We will prove results bounding its translation length \eqref{eq:poly_translation_length} by its Mahler measure (Theorem \ref{thm:sarasbounds_polyversion}), length, and discriminant (Corollary \ref{cor:sarabounds}).
By specializing these results to the case when $p(X)=p_x(X)$ is the characteristic polynomial for some semisimple $x\in\SL_n(\R)$, we obtain Theorem \ref{thm:sarasbounds}. 

We remind the reader that the \textit{Mahler measure} of $p$ is the product 
\begin{equation}
\displaystyle M(p):=\prod_{i=1}^n \max\{1,|a_i|\}.
\end{equation}
For a detailed survey on the Mahler measure, see \cite{smyth}.
While it is open as to whether the Mahler measure is uniformly bounded away from 1 for monic, irreducible, noncyclotomic polynomials in $\Z[X]$, we do have such a result when degree is fixed.

%
%

\begin{prop}\label{prop_mahler_monotic_lower_bounds}
There exists an explicit monotonically decreasing function on the domain of integers $n\ge2$, such that $f(n)>1$ and if $p(X)\in \Z[X]$ is a degree $n$, monic polynomial for which at least one of its irreducible factors is noncyclotomic, then $\log(M(p))\ge f(n)$.  Furthermore, $f(n)$ can be taken to be
\begin{equation}\label{eq:mahler_bounds}
f(n) = 
\begin{cases} 
\log(M(L))> 0.16235& \mbox{ if } n\le 20,\\
\frac{1}{4}\left(\frac{\log\log n}{\log n}\right)^3& \mbox{ if } n>20,
\end{cases}
\end{equation}
where $L(X) = X^{10}+X^9-X^7-X^6-X^5-X^4-X^3+X+1$ is Lehmer's polynomial.

\end{prop}

\begin{proof}
Lower bounds for when $p(X)$ is irreducible are well known, such as the bounds of Voutier \cite{V}
$f_{V}(n) =\frac{1}{4}\left(\frac{\log\log n}{\log n}\right)^3$.
 A check via Sage verifies that $f_{V}(15)<0.01245$, and $f_V(n)$ is monotonically decreasing for all $n\ge 15$.  
Meanwhile, Boyd \cite{Boyd1,Boyd2} enumerated all integer polynomials of degree $\le 20$ of small Mahler measure and verified that Lehmer's polynomial $L$ attains the minimal Mahler measure in this range.  
It can be computed in Sage that $\log(M(L))> 0.16235$.
We may take $f(x)$ to be the maximum of the two. 
When $p$ is not irreducible, the proposition follows from the fact that the log of the Mahler measure is additive together with the monotonicity of $f(n)$.
More precisely, if $p(X)=g_1(X)\cdots g_s(X)$, where the $g_i$ are irreducible of degree $n_i$, and without loss of generality, $g_1(X)$ is noncyclotomic, then 
$$\log(M(p))=\log(M(g_1))+\cdots +\log(M(g_s))\ge f(n_1)+\cdots +f(n_s)\ge f(n_1)\ge f(n).$$\end{proof}


We may write $p(X)$ as
\begin{align}\label{charpoly}
p(X) 
&=\sum_{j=0}^n (-1)^{j} s_jX^{n-j}
\end{align}
where $s_j$ denotes the $j^{th}$ \textit{elementary symmetric polynomial} in the roots of $p$. 
Let $p_k:=\sum_i a_i^k$. 
Newton's identities  \cite{Kalman} state that for all $1\le j\le n$, 
\begin{align}\label{newton}
%
js_j
=\sum_{i=1}^j (-1)^{i-1}s_{j-i} p_i.
\end{align}
%

The following result will be used in the proof of Theorem \ref{thm:sarasbounds} to bound the Mahler measure.
\begin{lem}\label{s1s2ineq} Suppose the polynomial $p(X)=\prod_{i=1}^n (X-c_i)=\sum_{j=0}^n (-1)^js_j X^{n-j}$ has all real zeros. The coefficient $s_j$ is the $j$th elementary symmetric polynomial in the zeros of $p$.  For $k\in\N$, let $p_k=\sum_{i=1}^n c_i^k$. Then: \begin{equation} 
 \frac{2n}{n-1}s_2\leq s_1^2\leq  n p_2
 \end{equation}
with equality on both sides exactly when $c_i=c_j$ for all $i,j$. If all $c_i\geq0$, then $p_2\leq s_1^2$ with equality exactly when $c_i=0$ for all $i$. 
\end{lem}

\begin{proof} By the Cauchy-Schwarz inequality:
$$ s_1^2=\left(\sum_{i=1}^n (c_i)(1)\right)^2\leq \left(\sum_{i=1}^n c_i^2\right) \left(\sum_{i=1}^n 1^2\right)=n \sum_{i=1}^n c_i^2=np_2 ,$$
with equality exactly when the sequence $\{c_i\}_{i=1}^n$ and $\{1\}_{i=1}^n$ are linearly dependent. In particular, when $c_i=c_j$ for all $i,j$. Also, 
$$s_1^2=\left(\sum_{i=1}^n c_i\right)^2=\sum_{i=1}^n c_i^2+2\sum_{i<j} c_ic_j=p_2+2s_2.$$
If all $c_i\geq 0$, notice that $s_1^2\geq p_2$ with equality precisely when all $c_i=0$. Rearranging the previous equation and combining it with the special case of the Cauchy-Schwarz inequality given above we get:
$$ s_1^2\leq n (s_1^2-2s_2) \quad\Rightarrow\quad 2ns_2\leq (n-1)s_1^2.$$
\end{proof} 

Being motivated by thinking of $p$ as a characteristic polynomial of a semisimple matrix, we define the \textit{translation length} of the polynomial $p$ to be

\begin{equation}\label{eq:poly_translation_length}
\displaystyle \ell(p):=\sqrt{\sum_{i=1}^n 2(\log\abs{a_i})^2}.
\end{equation}


\begin{thm}\label{thm:sarasbounds_polyversion}
Let $p(X)=\prod_{i=1}^n(X-a_i)\in \R[X]$ and assume $\prod_{i=1}^na_i=1$. Then:
\begin{equation}\label{eq:MahlerTrans}
2\sqrt{\frac{2}{n}}\log(M(p))
\leq\ell(p)
\leq 2 \log(M(p)).  \end{equation}
The lower bound is an equality exactly when there exists an $r\in\R$ such that for all $i$, either $|a_i|=r$ or $|a_i|^{-1}=r$. 
The upper bound is an equality when $\abs{a_i}=1$ for all but at most two of the $i$.
\end{thm}

\begin{rem}When $n=2$, the upper and lower bounds in \eqref{eq:MahlerTrans} are equal and so $\ell(x)=2\log(M(p))$, which coincides with \eqref{eq:original_mahler_length}. 
\end{rem}

\begin{proof}Let $p_{-1}(X):= \prod_{i=1}^n(X-a_i^{-1})$.  
Since $\prod_{i=1}^na_i=1$, $M(p)=M(p_{{-1}})>0$ and so $M(p)=\sqrt{M(p)M(p_{{-1}})}$. Then:
$$\log (M(p))
=\frac{1}{2}\log (M(p)M(p_{{-1}}))
=\frac{1}{2}\sum_{i=1}^n\left(\log_+\abs{a_i}+\log_+\abs{a_i^{-1}}\right)
=\frac{1}{2}\sum_{i=1}^n \abs{\log\abs{a_i}}  ,$$
where $\log_+(a)=\max\{0,\log(a)\}$. Applying Lemma \ref{s1s2ineq} to the sequence with $\{c_i\}_{i=1}^n=\{\abs{\log\abs{a_i}} \}_{i=1}^n$, we get: 
\begin{align}
2\log(M(p))=s_1
&\leq \sqrt{np_2}=\sqrt{n \sum_{i=1}^n \abs{\log\abs{a_i}}^2}  
= \sqrt{\frac{n}{2}}\ell(p), 
\end{align}
with equality precisely when, for all $i$, $\abs{\log\abs{a_1}}=\abs{\log\abs{a_i}}$ or, equivalently, when $\abs{a_i}\in\{\abs{a_1},\abs{a_1}^{-1}\}$.
 
We now find an upper bound for $\ell(p)$ by using specific information about the relative sizes of the $a_i$ that comes from our setup. Assume that $\abs{a_i}>1$ for some $i$, since otherwise $\abs{a_i}=1$ for all $i$ and so $\ell(p)=0=\log(M(p))$. Re-order the $\{a_i\}$ so that $\abs{a_1}\geq\ldots\geq \abs{a_k}\geq 1> \abs{a_{k+1}}\geq \ldots\geq \abs{a_n}>0$ and so $\prod_{i=1}^k \abs{a_i}=M(p)=M(p_{{-1}})=\prod_{i=k+1}^n \abs{a_i^{-1}}$. 
Using the same notation as in Lemma \ref{s1s2ineq}, we want to find a tight upper bound on $p_2$ in terms of $s_1$. We know that $p_2=s_1^2-2s_2\leq s_1^2$ since $c_i:=\abs{\log\abs{a_i}}\geq0$. To get a tighter bound, we find a lower bound on $s_2$ specific to our setup:
$$s_2=\sum_{i<j} c_ic_j
\geq \left(\sum_{i=1}^k c_i\right)\left(\sum_{j=k+1}^n c_j\right)
=\left(\log(M(p))\right)\left( \log(M(p_{{-1}}))\right)
=\left(\log(M(p))\right)^2.$$
Note that the above inequality is an equality when $c_i=0$ (so $\abs{a_i}=1$) for all $i\in[2,n-1]$.
Therefore:
\begin{align}
\ell(p)
&= \sqrt{2(s_1^2-2s_2)}
\leq \sqrt{2\left(2\log(M(p))\right)^2-4\left(\log(M(p))\right)^2}
=2 \log(M(p)).
\end{align}
\end{proof}

%
%

We conclude this section by observing that our results enable us to also bound translation length in terms of two other measures of a polynomial's complexity: length and discriminant.  
For $p(X)$ as above, its \textit{length} is:
\begin{equation}
\displaystyle L(p):=\sum_{i=0}^n \abs{s_{i}}
\end{equation} 
and its \textit{discriminant} is
\begin{equation}
\displaystyle \disc(p):=(-1)^{\frac{1}{2}n(n-1)}\prod_{i\neq j} (a_i-a_j)=\prod_{i<j} (a_i-a_j)^2.
\end{equation}

These quantities are related, and the following bounds are well known (see \cite{Mahler0})

\begin{align}\label{bddL&M}L(p)\le 2^nM(p)\le 2^nL(p)\end{align}

\begin{align}\label{discM}\abs{\disc(p)}\le n^nM(p)^{2n-2}\end{align}
 
 
 These inequalities can be combined with the inequality in Theorem \ref{thm:sarasbounds_polyversion} to bound $\ell(p)$ in terms of $L(p)$ and $\disc(p)$ as in the following corollary. Let $\log_+(a)=\max\{0,\log(a)\}$. 

 \begin{cor}\label{cor:sarabounds}
If $p(X)=\prod_{i=1}^n(X-a_i)\in \R[X]$ and $\prod_{i=1}^na_i=1$, then:
 \begin{align*}
  2\sqrt{\frac{2}{n}}\log_+(2^{-n}L(p))&\leq\ell(p)\leq 2 \log(L(p)),\\
   \frac{1}{n-1}\sqrt{\frac{2}{n}}\log_+\abs{n^{-n}\disc(p)}&\leq\ell(p).
 \end{align*}
 \end{cor}
 
\begin{proof}
 The length inequality follows directly from replacing $M$ in Theorem \ref{thm:sarasbounds_polyversion} with $L$ using \eqref{bddL&M}.
 To get the discriminant inequality, first re-arrange \eqref{discM} and take the logarithm to get: 
 $$\log_+\abs{n^{-n}\disc(p)}\leq 2(n-1)\log(M(p)).$$
Combine that with the lower bound in Theorem \ref{thm:sarasbounds_polyversion} that $2\log(M(p))\leq \sqrt{\frac{n}{2}} \ell(p)$ to get: 
$$\log_+\abs{n^{-n}\disc(p)}\leq (n-1)\sqrt{\frac{n}{2}} \ell(p).$$
 \end{proof}

\section{Characteristic Polynomials and the Proof of Theorem  \ref{thm:universalbound}}\label{section:universal_bounds}

If $\mathfrak{X}$ is a globally symmetric space of noncompact type, then 
the identity component of its isometry group $\mathcal{G}:=\isom_0(\mathfrak{X})$ is a connected, adjoint semisimple Lie group and 
the stabilizer subgroup $\mathcal{K}=\stab_{\mathcal{G}}(o)$ of a point $o\in \mathfrak{X}$ is a maximal compact subgroup \cite[IV.3.3, VI.1.1, VI.2.2]{H}.
It follows that $\mathcal{K}\backslash\mathcal{G}\cong \mathfrak{X}$.
We say $\mathfrak{X}$ is \textit{simple} if $\mathcal{G}$ is simple (i.e., the complexification of its Lie algebra is simple). 
By a \textit{simple locally symmetric space}, we mean a space of the form $\mathfrak{X} / \Gamma$ where $\mathfrak{X}$ is simple and $\Gamma \subset \mathcal{G}$ is a lattice.

We now discuss the construction of arithmetic lattices in $\mathcal{G}$.  
Here we assume some familiarity with algebraic groups and arithmetic groups.
%
%
For detailed references on algebraic and arithmetic groups, we refer the reader to \cite{Borel, Witte}.
For a discussion of arithmetic simple orbifolds, see \cite[Section 7]{LLM}.

To begin, fix:
\begin{enumerate}[label={(\arabic*)}]
\item a degree $d_1$, totally real number field $k$ with ring of integers $\mathcal{O}_k$,
\item an embedding $k\subset \R$,
\item a simple, semisimple, adjoint, algebraic $k$-group $\mathrm{G}$ of dimension $d_2$ such that the identity component of its real points $(\mathrm{G}(\R))_0\cong \mathcal{G}$ as Lie groups,
\item a choice of a $k$-rational basis for its Lie algebra $\mathfrak{g}$ and use it to identify $\GL(\frak{g})$ with $\GL_{d_2}(k)$,
\item $\mathrm{G}(\mathcal{O}_k):=\Ad^{-1}(\Ad(\mathrm{G}(k)) \cap \GL_{d_2}(\mathcal{O}_k))$.
\end{enumerate}

Assume $\mathrm{G}$ is $\R$-isotropic but $\mathrm{G}$ is anisotropic for all non-identity embeddings $\sigma:k\to \R$.  
Then any $\Gamma\subset \mathrm{G}(\R)$ commensurable to $\mathrm{G}(\mathcal{O}_k)$ is an \textit{arithmetic lattice} in $\mathrm{G}(\R)$.  

Fix a basis for $k$ over $\Q$ and let $\Reg:k\to\Mat_{d_1}(\Q)$ be the regular representation associated to this basis.
Let $L/k$ be the Galois closure of $k$ with fixed embedding $L\subset \C$.
We have two injective ring maps $\iota_1,\iota_2: \Mat_{d_2}(k) \to \Mat_{d_1d_2}(L)$ defined below.
We obtain our first ring map by applying the regular representation to each entry of the matrix.
$$\iota_1: \Mat_{d_2}(k)\to \Mat_{d_1d_2}(\Q)\subset  \Mat_{d_1d_2}(L)$$

$$
\begin{pmatrix}
y_{11} & \cdots & y_{1d_2}\\
	& \ddots & \\
y_{d_21} & \cdots & y_{d_2d_2}
\end{pmatrix}
\mapsto
\left(\begin{matrix}[c|c|c]
\Reg(y_{11}) & \cdots & \Reg(y_{1d_2})\\ \hline 
	& \ddots & \\ \hline 
\Reg(y_{d_21}) & \cdots & \Reg(y_{d_2d_2})
\end{matrix}\right).$$

Meanwhile, we obtain our second injective ring map by looking at the diagonal map across all real embeddings of $k$, which we denote $V_\infty$.  
 
$$\iota_2: \Mat_{d_2}(k)\to \prod_{\sigma\in V_{\infty}} \Mat_{d_2}(L)\subset \Mat_{d_1d_2}(L)$$

$$
\begin{pmatrix}
y_{11} & \cdots & y_{1d_2}\\
	& \ddots & \\
y_{d_21} & \cdots & y_{d_2d_2}
\end{pmatrix}
\mapsto
\left(\begin{array}{@{}c@{}cc@{}}
\hspace{0.5pc}  
\begin{array}{|ccc|}\hline
\sigma_1(y_{11}) & \cdots & \sigma_1(y_{1d_2})\\
	& \ddots & \\
\sigma_1(y_{d_21}) & \cdots & \sigma_1(y_{d_2d_2})\\\hline
  \end{array} \hspace{0.5pc}
  & 0& 0 \\
  0 & \ddots & 0\\
  0 &  0 &
\hspace{0.5pc}  
\begin{array}{|ccc|}\hline
\sigma_{d_1}(y_{11}) & \cdots & \sigma_{d_1}(y_{1d_2})\\
	& \ddots & \\
\sigma_{d_1}(y_{d_21}) & \cdots & \sigma_{d_1}(y_{d_2d_2})\\\hline
  \end{array} \hspace{0.5pc}
\end{array}\right).
 $$
 
 \begin{lem}\label{restriction_conjugate}
The images $\iota_1(\Mat_{d_2}(k))$ and $\iota_2(\Mat_{d_2}(k))$ are $\GL_{d_1d_2}(L)$-conjugate.
\end{lem}
\begin{proof}
If $A:=L\otimes_k \Mat_{d_2}(k)$, then $A$ is a central simple $L$-algebra \cite[7.8]{Reiner}.  
Letting $B:= \Mat_{d_1d_2}(L)$, and letting $\iota_1$ and $\iota_2$ extend linearly to $A$, we have $L\subset \iota_1(A)\subset B$ and $L\subset \iota_2(A)\subset B$.
It then follows from the Skolem--Noether Theorem \cite[7.21]{Reiner} that the images $\iota_1(A)$ and $\iota_2(A)$
are identified by an inner automorphism of $A$.
More precisely, there exists a $g \in \GL_{d_1d_2}(L)$ such that
if $h \in \Mat_{d_2}(k)$, then $\iota_2(1\otimes h) = g\iota_1(1\otimes h)g^{-1}$.  The result follows.
%
\end{proof}

If $\mathrm{H}$ is an algebraic $k$-group, the restriction of scalars group $\Res_{k/\Q}(\mathrm{H})$ is a $\Q$-group such that $\mathrm{H}(k)\cong \Res_{k/\Q}(\mathrm{H})(\Q)$ \cite[2.1.2]{PlatRap}.
Fix the $\Q$-rational map $\varphi: \Res_{k/\Q}(\SL_{d_2})\to \SL_{d_1d_2}$ which at the level of $k$-points is $\iota_1$, i.e.
if $y \in \SL_{d_2}(k)\cong  \Res_{k/\Q}(\SL_{d_2})(\Q)$, then $\varphi(y)=\iota_1(y)$.

As in \cite[Proposition 7.2]{LLM}, composing the adjoint map with restriction of scalars, we obtain an embedding of Lie groups $\rho:\mathrm{G}(\R)\to \SL_{d_1d_2}(\R)$.
This embedding has many nice features and in Proposition \ref{prop:rep_char_poly}  we now establish an important one concerning characteristic polynomials: relative to this embedding, semisimple elements of arithmetic lattices always have characteristic polynomials with integer coefficients.

\begin{prop}\label{prop:rep_char_poly}
If $\Gamma\subset \mathrm{G}(\R)$ is an arithmetic lattice and $x\in \Gamma$ is semisimple, then $p_{\Ad(x)}(X)\in\mathcal{O}_k[X]$ and $p_{\rho(x)}(X)\in \Z[X]$.
\end{prop}

\begin{proof}
To begin, we show that if $x\in \Gamma$ is semisimple, then the characteristic polynomial $p_{\Ad(x)}(X)$ of the matrix $\Ad(x)\in \SL_{d_2}(k)$ lies in $\mathcal{O}_k[X]$.
By assumption, $\Gamma$ is commensurable with $\mathrm{G}(\mathcal{O}_k)$, so let $e=[\Gamma:(\Gamma\cap \mathrm{G}(\mathcal{O}_k))]$.  
Then $x^e\in \mathrm{G}(\mathcal{O}_k)$, and thus $p_{\Ad(x^e)}(X)\in \mathcal{O}_k[X]$.
Hence $\Ad(x^e)$ has integral eigenvalues, and thus the eigenvalues of $\Ad(x)$ are algebraic integers.  
The coefficients of the characteristic polynomial are symmetric polynomials on the eigenvalues \eqref{charpoly}, thus the coefficients of $p_{\Ad(x)}(X)$ are algebraic integers.
Since $\mathrm{G}$ is adjoint, $x \in \mathrm{G}(k)$  \cite[Proposition 1.2]{BP}, so $\Ad(x)\in \SL_{d_2}(k)$ hence $p_{\Ad(x)}(X)\in k[X]$, and we conclude $p_{\Ad(x)}(X)\in \mathcal{O}_k[X]$.

Next we show that if $y\in \SL_{d_2}(k)$ is semisimple such that $p_y(X)\in\mathcal{O}_k[X]$, then $p_{\iota_1(y)}(X)\in \Z[X]$.
To begin, since $y\in \SL_{d_2}(k)$, it follows that $\iota_1(y)\in \SL_{d_1d_2}(\Q)$ and hence $p_{\iota_1(y)}(X)$ has rational coefficients.
Meanwhile $\iota_2$, restricted to $\SL_{d_2}(k)$, gives the diagonal embedding
$$\SL_{d_2}(k)\to \prod_{\sigma\in V_{\infty}} \SL_{d_2}(\sigma (k))\subset \SL_{d_1d_2}(L)$$ 
where $V_{\infty}$ is the set of all real embeddings of $k$ and $L$ is the Galois closure of $k$.
Since $\iota_1(y)$ and $\iota_2(y)$ are $\GL_{d_1d_2}(L)$-conjugate by Lemma \ref{restriction_conjugate}, they have equal characteristic polynomials.
If $p_y(X) = \sum_{i=0}^nc_iX^i$ for $c_i\in \mathcal{O}_k$, then it follows that $$p_{\iota_2(y)}(X) = \prod_{\sigma\in V_{\infty}}\left(\sum_{i=0}^n\sigma(c_i)X^i\right).$$
Since each $\sigma(c_i)$ is an algebraic integer, all of the coefficients of $p_{\iota_2(y)}(X)$ are algebraic integers. 
We deduce that the coefficients of $p_{\iota_1(y)}(X)$ are both rational and algebraic integers, thus concluding $p_{\iota_1(y)}(X)\in \Z[X]$.
\end{proof}

While there are many reasonable Riemannian metrics to put on $\mathfrak{X}_n:=\SO(n)\backslash \SL_n(\R)$, as in \cite[Equation (10)]{LLM}, we fix once and for all the Riemannian metric on 
$\mathfrak{X}_n$ given by
\begin{equation}\label{metric_eq}
\langle X,Y\rangle:= 2\tr(XY),\hspace{2pc} X,Y\in\mathfrak{sl}_n(\R).
\end{equation}
This has the nice geometrical property that the inclusions
of the hyperbolic plane $\mathbb{H}^2\to \mathfrak{X}_n$
(associated to any $2\times 2$ block along the diagonal $\SL_2(\R)\to \SL_n(\R)$) are totally geodesic, isometrical immersions (i.e., the images have constant curvature -1).

For any inclusion of a connected semisimple lie group $H\to \SL_n(\R)$, this induces a \textit{subspace metric} (which we will denote $g$) on the corresponding symmetric space $(H\cap \SO(n)) \backslash H \subset \mathfrak{X}_n$.

\begin{proof}[Proof of Theorem \ref{thm:universalbound}]
Let $\Gamma\subset \mathrm{G}(\R)$ be an arithmetic subgroup
%
%
%
and $N:=\mathrm{K}\backslash \mathrm{G}(\R)/\Gamma$, where $K$ is a maximal compact subgroup of $\mathrm{G}(\R)$.
By \cite[Theorem 7.2]{LLM}, we have an embedding of Lie groups
$\rho:\mathrm{G}(\R)\to \SL_{n}(\R)$ where $n=|k:\Q| \dim \mathrm{G}$.
It follows from Proposition \ref{prop:rep_char_poly} that if  $x\in\Gamma$ is semisimple, then $p_{\rho(x)}(X)\in \Z[X]$, and hence 
by our Mahler-Length Theorem \ref{thm:sarasbounds} and Proposition \ref{prop_mahler_monotic_lower_bounds}, we have 
\begin{equation}
\sys(N,g)\ge \frac{2\sqrt{2}}{\sqrt{n}}f(n),
\end{equation}
where $\sys(N,g)$ denotes the systole of $N$ relative to the subspace metric \eqref{metric_eq}, and $f(n)$ is as in \eqref{eq:mahler_bounds}.
\end{proof}

\section{Large Systoles and the Proof of Theorem \ref{thm:magnitude}}\label{section:large_systole}

Let $N$ be a compact, simple arithmetic manifold. In \cite[Theorem B]{LLM} we showed that $N$ has a finite degree cover $N'$ whose systole satisfies $\sys(N')\ge c_1\log(\Vol(N'))+c_2$ for certain positive constants $c_1$ and $c_2$. In order to prove Theorem 
\ref{thm:magnitude} it therefore suffices to derive an upper bound for the systole of $N'$ whose order of magnitude is also $\log(\Vol(N'))$. This follows from the following proposition.


\begin{prop}\label{systoleupperbound}
Let $\mathfrak X$ be a symmetric space of noncompact type and $\Gamma$ be a discrete, cocompact subgroup of isometries of $\mathfrak X$ having finite covolume. Then $\sys(\mathfrak X/\Gamma)\leq c \log(\Vol(\mathfrak X/\Gamma))$ for some positive constant $c$ depending only on $\mathfrak X$.
\end{prop}
\begin{proof}
Let $B$ be a closed ball of radius $R$ in $\mathfrak X$ which has volume greater than $\Vol(\mathfrak X/\Gamma)$. The projection $\pi: B\rightarrow \mathfrak X/\Gamma$ cannot be one-to-one, hence there exist distinct points $x_1,x_2\in B$ such that $\pi(x_1)=\pi(x_2)$. Then $x_2=\gamma x_1$ for some semisimple element $\gamma\in\Gamma$, and consequently the translation length $\ell(\gamma)$ of $\gamma$ satisfies $\ell(\gamma)\leq 2R$. The proposition now follows from \cite[Theorem A]{knieper}, where it is shown that the volume of $B$ is asymptotically equivalent to $R^{\frac{r-1}{2}}e^{h R}$ as $R$ tends to infinity, where $r$ is the rank of $\mathfrak X$ and $h>0$ denotes the volume entropy of $\mathfrak X$. Indeed, the result of \cite[Theorem A]{knieper} shows that the volume of $B$ is greater than $c_1e^{hR}$ for all $R\geq R_0$, where $c_1$ and $R_0$ are positive constants. In particular the volume of $B$ is will be greater than $\Vol(\mathfrak X/\Gamma)$ for any $R \geq R_0$ satisfying $R> \frac{1}{h}\log(\Vol(\mathfrak X/\Gamma)/c_1)$, and there clearly exists a positive constant $c$ for which such an $R$ may be taken to satisfy $R<c\log(\Vol(\mathfrak X/\Gamma))$.\end{proof}

%
%
%
%






\section{Geometry of Special Linear Orbifolds and Theorem \ref{thm:systolebounds}}\label{section:speciallinearbackground}
%

For $n\ge 2$, we fix the globally symmetric space $\mathfrak{X}_n:=\SO(n) \backslash \SL_n(\R)$.
If $\Gamma \subset \SL_n(\R)$ is a lattice, then $\mathfrak{X}_n/\Gamma$ is a \textit{special linear orbifold}.
In this section, we give a brief overview of the construction and geometry of what we call \textit{standard} special linear orbifolds.
For a detailed construction, we refer the reader to 
\cite[Section 4]{LLM}.
This construction requires knowledge of the theory of central simple algebras, for which we refer the reader to \cite{Reiner}.

Let $A$ be a central simple algebra over $\Q$ of dimension $n^2\ge 4$.
The Wedderburn Structure Theorem states that $A\cong \mathrm{M}_m(D)$ for some division algebra $D$ over $\Q$ where $n^2=m^2(\dim D)^2$.

Let $\ell$ denote either $\R$ or the field of $p$-adic numbers for some rational prime $p$.
Then $A\otimes_\Q \ell$ is a central simple algebra over $\ell$, and again by the Wedderburn Structure Theorem, there exists a division algebra $D_\ell$ over $\ell$ such that there exists an isomorphism
$h:A\otimes_\Q \ell \cong \mathrm{M}_{m'}(D_\ell)$.
If $\dim D_{\ell}>1$, we say $A$ ramifies over $\ell$, otherwise we say $A$ is unramified over $\ell$.

A $\Z$-\textit{order} in $A$ is a subring $\mathcal{O}$ which is also a finitely generated $\Z$-module for which $A=\Q\otimes_\Z \mathcal{O}$.
An order is \textit{maximal} if it is maximal with respect to inclusion.  Maximal orders always exist \cite[10.4]{Reiner}.
Let $\mathcal{O}$ be a maximal order.

To produce our standard special linear manifolds, we fix a central simple algebra $A$ over $\Q$ that is unramified over $\R$.
Fix the isomorphism $h: A\otimes_\Q \R \to \mathrm{M}_n(\R)$.
Via that isomorphism, we can define the \textit{norm}
$N(a):=\det(h(a\otimes 1))$.
(It is a fundamental result that this definition is independent of the choice of isomorphism $h$).

Let $\mathcal{O}^1$ denote the set of norm one elements of $\mathcal{O}$.
Following the convention of \cite[8.3.4]{MaR}, if $\Gamma \subset h(\mathcal{O}^1)$, then we say $\Gamma$ is a 
\textit{standard arithmetic lattice in $\SL_n(\R)$ derived from a central simple algebra} and similarly
$M:= \mathfrak{X}_n / \Gamma$ is a \textit{standard special linear orbifold derived from a central simple algebra}.
Any orbifold commensurable to such an $M$ is a \textit{standard linear orbifold}.


As a locally symmetric space, $M$ has a one parameter family of Haar measures induced from $\mathfrak{X}_n$.
We choose the normalization \eqref{metric_eq} which we called the \textit{geometric metric} \cite[Section 2]{LLM}, in which the natural embedding $\SL_2(\R)\to \SL_n(\R)$ as a $2\times 2$ block along the diagonal results in a totally geodesic, isometric embedding $\mathbb{H}^2 \to \mathfrak{X}_n$.

%
%

\begin{proof}[Proof of Theorem \ref{thm:systolebounds}]
Let $M$ be a special linear orbifold derived from a central simple algebra., i.e. $M \cong \mathfrak{X}_n / \Gamma$ where $\Gamma \subset \mathcal{O}^{1}$, the norm 1 elements of a maximal order $\mathcal{O}$.
Let $x\in \Gamma$ be semisimple and let $p_x(X)$ denote the characteristic polynomial of $h(x\otimes 1)$ for $h$ as above.
It is a fundamental result \cite[Theorem 8.6]{Reiner} that 
$p_x(X)$ is a monic, degree $n$ polynomial in $\Z[X]$.
(Note that $p_x$ need not be irreducible.)
By our Mahler-Length Theorem \ref{thm:sarasbounds} and Proposition \ref{prop_mahler_monotic_lower_bounds}, we have then
$\sys(M)\ge \frac{2\sqrt{2}}{\sqrt{n}}f(n),$
where $f(n)$ is as in \eqref{eq:mahler_bounds}, and the result follows.
\end{proof}

\section{Proof of Theorem \ref{thm:effectivegw}}

Let $p>2$ be a rational prime, $k$ be a totally real number field with ring of integers $\mathcal O_k$ and  $\mathfrak{p}_1,\ldots \mathfrak{p}_s$ be primes of $k$. 

Let $B$ be the division algebra over $k$ with Hasse invariant $\frac{1}{p}$ at each of the primes $\mathfrak p_1,\dots, \mathfrak p_{s-1}$ and Hasse invariant $\frac{a}{p}$ at the prime $\mathfrak p_s$ where $a$ is a positive integer satisfying $a + (s-1) \equiv 0 \pmod{p}$. Since the sum of these Hasse invariants is an integer, it follows from the Albert-Brauer-Hasse-Noether theorem that $B$ exists and has degree $p$ over $k$ (see \cite[Theorem 2.1]{LMPT}.

Let $\mathcal{O}$ be a maximal order of $B$ and $\mathcal{O}^1$ denote the multiplicative subgroup of $\mathcal O^*$ consisting of the elements of $\mathcal O^*$ having reduced norm one. Since $B\otimes_\Q \R\cong \MM_p(\R)^{[k:\Q]}$ we obtain an embedding $B\hookrightarrow \SL_p(\R)^{[k:\Q]}$. Let $\Gamma$ denote the image in $\SL_p(\R)^{[k:\Q]}$ of $\mathcal O^1$ under this embedding. Let $\mathfrak X$ denote the symmetric space of $\SL_p(\R)^{[k:\Q]}$ and define $M:= \mathfrak X/\Gamma$ to be the arithmetic orbifold associated to $\Gamma$.

Let $\gamma\in \Gamma$ be a semisimple element whose translation length $\ell(\gamma)$ realizes the systole of $M$. As a consequence of Proposition \ref{systoleupperbound} we have the inequality

\begin{equation}\label{upperbound1}
\ell(\gamma) \leq c_1\log(\Vol(\mathfrak X/\Gamma))
\end{equation}

for some positive constant $c_1$ depending only on $\mathfrak X$ (and thus only on $p$ and the degree of $k$). 

We now obtain an upper bound for $\Vol(\mathfrak X/\Gamma)$ using \cite[Proposition 3.2]{LMM}:

\begin{equation}\label{upperbound2}
\Vol(\mathfrak X/\Gamma) \leq c_2 d_k^{\frac{p^2-1}{2}} \left(\prod_{i=1}^s N_{k/\Q}(\mathfrak p_i)\right)^{\frac{p(p-1)}{2}}
\end{equation}
 
for some positive constant $c_2$ which is explicitly given in \cite[Proposition 3.2]{LMM} and depends only on $p$ and the degree of $k$. (That the dependence only depends on the degree of $k$ follows from the trivial bound $\zeta_k(s) \leq \zeta(s)^{[k:\Q]}$.)

We now need a lower bound for $\ell(\gamma)$. To do so we will employ the bound given in Theorem \ref{thm:sarasbounds}. There is a potential issue however. The hypothesis of Theorem \ref{thm:sarasbounds} is that the semisimple element $\gamma$ lies in $\SL_n(\R)$ for some positive integer. Our element $\gamma$, on the other hand, lies in $\SL_p(\R)^{[k:\Q]}$. To resolve this we use the fact that $\SL_p(\R)^{[k:\Q]}$ is a Lie subgroup of $\SL_{p[k:\Q]}(\R)$ and apply Theorem \ref{thm:sarasbounds} with $n=p[k:\Q]$:

\begin{equation}\label{lowerbound1}
\ell(\gamma) \geq 2\sqrt{\frac{2}{p[k:\Q]}} \log(M(p_\gamma)).
\end{equation}

Let $L=k(\gamma)$ be the extension of $k$ generated by adjoining to $k$ a preimage in $B$ of $\gamma$. This extension has degree $p$ and, by Theorem 32.15 of \cite{Reiner}, for every prime $\mathfrak P$ of $L$ lying above a prime $\mathfrak p\in\{\mathfrak p_1,\dots, \mathfrak p_s\}$, satisfies $[L_{\mathfrak{P}}:k_{\mathfrak p}]=p$. In particular none of the primes $\mathfrak p_i$ split completely in $L/k$.

In order to relate the relative discriminant $\Delta_{L/k}$ of $L$ over $k$ we will make use of a result of Silverman \cite[Theorem 2]{silverman}:

\begin{equation}\label{lowerbound2}
\log(M(p_\gamma)) \geq \frac{1}{2p(p-1)}\log(N_{k/\Q}(\Delta_{L/k})) - \frac{p\log(p)}{2(p-1)}.
\end{equation}

By combining (\ref{lowerbound1}) and (\ref{lowerbound2}) we obtain positive constants $c_3, c_4$ depending only on $p$ and the degree of $k$ such that

\begin{equation}\label{lowerbound3}
\ell(\gamma) \geq c_3+c_4\log(N_{k/\Q}(\Delta_{L/k})).
\end{equation}

Finally, by combining (\ref{upperbound1}), (\ref{upperbound2}) and (\ref{lowerbound3}) we obtain positive constants $c_5,c_6,c_7$ such that 

\[N_{k/\Q}(\Delta_{L/k}) \leq c_5 d_k^{c_6} \left(\prod_{i=1}^s N_{k/\Q}(\mathfrak p_i)\right)^{c_7},\]

which finishes the proof of Theorem \ref{thm:effectivegw}.

%
%
%
%



\end{document}